\documentclass[11pt,reqno]{amsart}

\usepackage[colorlinks=true, pdfstartview=FitV, linkcolor=blue, 
citecolor=blue]{hyperref}

\usepackage{amssymb,amsmath,amscd}
\usepackage{bbm}
\usepackage{xfrac}
\usepackage{graphicx}
\usepackage{a4wide}
\usepackage{amsaddr}
\usepackage{tikz}
\tikzset{every picture/.style={line width=0.75pt}}

\newtheorem{theorem}{Theorem}[section]
\newtheorem{prop}[theorem]{Proposition}
\newtheorem{lemma}[theorem]{Lemma}     
\newtheorem{fact}[theorem]{Fact}

\theoremstyle{definition}

\newtheorem{remark}[theorem]{Remark}

\newcommand{\ts}{\hspace{0.5pt}}
\newcommand{\nts}{\hspace{-0.5pt}}

\newcommand{\RR}{\mathbb{R}\ts}

\newcommand{\ZZ}{\mathbb{Z}}

\newcommand{\NN}{\mathbb{N}}

\newcommand{\QQ}{\mathbb{Q}}
\newcommand{\TT}{\mathbb{T}}
\newcommand{\XX}{\mathbb{X}}
\newcommand{\YY}{\mathbb{Y}}

\newcommand{\cK}{\mathcal{K}}
\newcommand{\cL}{\mathcal{L}}

\newcommand{\cO}{\mathcal{O}}

\newcommand{\vL}{\varLambda}
\newcommand{\vO}{\varOmega}
\newcommand{\vS}{\varSigma}

\newcommand{\oplam}{\mbox{\Large $\curlywedge$}}

\newcommand{\dd}{\ts\mathrm{d}}

\newcommand{\exend}{\hfill$\Diamond$}

\newcommand{\myfrac}[2]{\frac{\raisebox{-2pt}{$#1$}}
      {\raisebox{0.5pt}{$#2$}}}

\newcommand{\twovec}[2]{\bigl(\begin{smallmatrix} #1 \\ #2
      \end{smallmatrix}\bigr)}
\numberwithin{equation}{section}

\makeatletter
\renewcommand{\@captionfont}{\small}
\makeatother

\DeclareMathOperator{\dotcup}{\dot\cup}

\DeclareMathOperator{\dist}{dist}

\DeclareFontFamily{U}{mathx}{\hyphenchar\font45}
\DeclareFontShape{U}{mathx}{m}{n}{ <5> <6> <7> <8> <9> <10>
   <10.95> <12> <14.4> <17.28> <20.74> <24.88> mathx10 }{}
\DeclareSymbolFont{mathx}{U}{mathx}{m}{n}
\DeclareFontSubstitution{U}{mathx}{m}{n}
\DeclareMathAccent{\widecheck}{0}{mathx}{"71}

\newcommand{\defeq}{\mathrel{\mathop:}=}
\newcommand{\eqdef}{=\mathrel{\mathop:}}

\begin{document}

\title{Fibonacci direct product variation tilings}

\author{Michael Baake, Franz G\"{a}hler, Jan Maz\'a\v c}

\address{Fakult\"at f\"ur Mathematik, Universit\"at Bielefeld,
   Postfach 100131, 33501 Bielefeld, Germany}
\email{$\{$mbaake,gaehler,jmazac$\}$@math.uni-bielefeld.de}





\begin{abstract} 
  The direct product of two Fibonacci tilings can be described as a
  genuine stone inflation rule with four prototiles. This rule admits
  various modifications, which lead to 48 different inflation rules,
  known as the direct product variations. They all result in tilings
  that are measure-theoretically isomorphic by the Halmos--von Neumann
  theorem. They can be described as cut and project sets with
  characteristic windows in a two-dimensional Euclidean internal
  space. Here, we analyse and classify them further, in particular
  with respect to topological conjugacy.
\end{abstract}

\maketitle
\section{Introduction}

The structure determination of perfect crystals from their diffraction
image consists of two steps, namely the extraction of the underlying
lattice from the support of the Bragg peaks and then the
reconstruction of the atomic positions from the scattering
intensities, which is a tricky inverse problem.

In the case of perfect quasicrystals, the analogous steps have to be
performed in the setting of cut and project sets. Concretely, one has
to identify the embedding lattice from the support of the Bragg
spectrum and then the window from the intensities. While different
structures with the same space group, in the fully periodic setting,
are always locally derivable from one another (by a simple
re-decoration of the fundamental cell), this is way more complex in
the quasiperiodic scenario.

Here, we demonstrate some of the new phenomena along the direct
product of two Fibonacci tilings of the plane and their altogether
$48$ \emph{direct product variations} (DPVs). In particular, we
provide a~finer classification of them, by showing a strong
topological conjugacy for one subclass. Concretely, we show in
Theorem~\ref{thm:new} that the $28$ DPVs with polygonal windows,
respectively the dynamical systems induced by them, are topologically
conjugate to one another, even though they are generally not mutually
locally derivable from one another, in the sense of
\cite[Ch.~5.2]{TAO}. This extends and completes the results of
\cite{BFG}.

\section{1D Fibonacci tiling}
Let us begin with a brief review of the well-known \emph{Fibonacci
  substitution} in one dimension, which is given by the following
substitution rule over the binary alphabet $\{a,b\}$,
\[
    \varrho : \; a \mapsto ab \, , \; b \mapsto a \ts .
\]
This symbolic substitution can be turned into a geometric inflation
rule with two prototiles (intervals) $a$, $b$ of length
$\tau = \frac{1}{2} \bigl( 1 + \mbox{\small $\sqrt{5}$} \, \bigr)$ and
$1$, respectively. Note that these lengths are the entries of a left
eigenvector of the corresponding substitution matrix
\begin{equation}\label{eq:fibmat}
    M^{}_{\varrho} \, = \, \begin{pmatrix} 1 & 1 \\ 1 & 0 \end{pmatrix}
\end{equation}
for its \emph{Perron--Frobenius} (PF) eigenvalue, which is $\tau$.
Taking any bi-infinite fixed point of $\varrho$ (more precisely, of
$\varrho^2$ in this case), one obtains a tiling of the real line. If
one assigns to each tile a special control point (say the left
endpoint of each interval), one gets a discrete point set
$\vL^{}_{\mathrm{F}} =\vL^{a}_{\mathrm{F}} \cup
\vL^{b}_{\mathrm{F}}$. Moreover, this set is a Delone set and is
obviously \emph{mutually locally derivable} (MLD) with the given
tiling, see \cite{TAO} for background and further details. This allows
us to identify these two representations of the fixed point. More
generally, for any tiling, we will always have both representations in
mind in what follows.

Recall that the fixed point can be chosen so that
$0\in \vL^{}_{\mathrm{F}}$ and thus
$\vL^{}_{\mathrm{F}} \subset \ZZ[\tau]$. It is useful to describe
$\vL^{}_{\mathrm{F}}$ as a cut and project set. This description is
based on the Minkowski embedding of $\ZZ [\tau]$ as a lattice in
$\RR^2$, namely as
\[
  \cL^{}_{1} \, \defeq  \,
  \bigl\{ (x,x') : x \in \ZZ [ \tau] \bigr\}.  
\]
Here, $'$ denotes the non-trivial field automorphism of $\QQ(\tau)$,
which maps $\tau$ to its algebraic conjugate
$ \sigma = \sfrac{-1}{\tau} = 1-\tau $. Then, following standard
arguments \cite[Ch.~7]{TAO}, one gets
\[
    \vL^{a,b}_{\mathrm{\mathrm{F}}} \, = \, \{ x\in \ZZ [\tau] :
    x' \in \vO^{}_{a,b} \} \ts ,  
\]
with
\begin{equation}\label{eq:win1DfiboA}
    \vO^{}_{\,a} \, = \, [\tau-2,\tau-1)
    \quad \text{and} \quad
    \vO^{}_{\, b} \, = \, [-1,\tau-2)
\end{equation}   
or with
\begin{equation}\label{eq:win1DfiboB}
    \vO^{}_{\, a} \, = \, (\tau-2,\tau-1]
    \quad \text{and} \quad
    \vO^{}_{\, b} \, = \, (-1,\tau-2] \ts ,
\end{equation}  
depending on the choice of the fixed point of the substitution; see
\cite[Ex.~7.3]{TAO}.

\section{2D Fibonacci tiling and its variations}

Having the 1D Fibonacci substitution, one can apply it in two
different directions in a~plane, say along the standard coordinate
axes. Considering all Cartesian products of tiles in the two
directions results in an inflation tiling of the plane with four
prototiles $T_0$, $T_1$, $T_2$, $T_3$ and
\begin{equation}\label{eq:fibosqrule}
  \raisebox{-14pt}{\includegraphics[width=0.8743\textwidth]
    {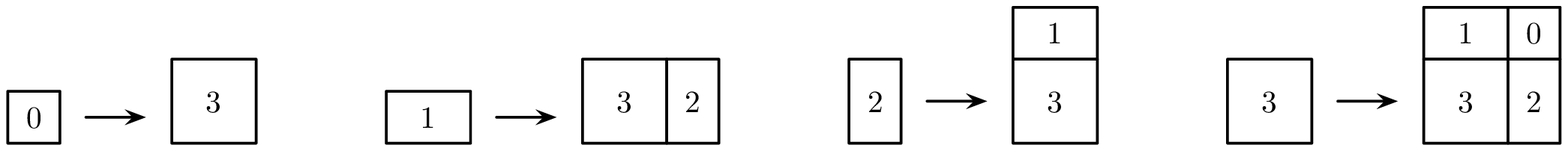}}
\end{equation}
which is stone inflation rule. Let us call the emerging tiling a
\emph{Fibonacci direct product} (DP) \emph{tiling} in two
dimensions. Similarly, one can proceed further and define
higher-dimensional product tilings. In our case, the corresponding
substitution matrix $M$ is given by
\begin{equation}\label{eq:2subst}
    M \, = \, \begin{pmatrix} 
    0 & 0 & 0 & 1 \\
    0 & 0 & 1 & 1 \\
    0 & 1 & 0 & 1 \\
    1 & 1 & 1 & 1 
    \end{pmatrix}  \, = \, 
    \begin{pmatrix} 0 & 1 \\ 1 & 1 \end{pmatrix} \otimes
    \begin{pmatrix} 0 & 1 \\ 1 & 1 \end{pmatrix}.
\end{equation}
The matrix is primitive and the left eigenvector (corresponding to the
PF eigenvalue~$\tau^2$) can be chosen to be
$\bigl(1,\tau,\tau,\tau^2\bigr)$. The entries are the areas of the
prototiles $T_i$. Following the same procedure as above (choosing the
lower left corners of the prototiles as their control points), one
obtains a point set $\vL = \bigcup_{i=0}^{3}\vL^{}_i$ which is MLD
with the Fibonacci DP tiling. Here, the sets $\vL_i$ satisfy the
equations
\begin{equation}\label{eq:ssr}
\begin{split}
  \vL^{}_{0} &\,=\, \tau \vL^{}_{3} + \twovec{\tau}{\tau}\ts ,\\
  \vL^{}_{1} &\,=\, \tau \vL^{}_{2} + \twovec{0}{\tau} \,\dotcup\,
                   \tau \vL^{}_{3} + \twovec{0}{\tau}\ts ,\\
  \vL^{}_{2} &\,=\, \tau \vL^{}_{1} + \twovec{\tau}{0} \,\dotcup\,
                   \tau \vL^{}_{3} + \twovec{\tau}{0}\ts ,\\
  \vL^{}_{3} &\,=\, \tau \vL^{}_{0} \,\dotcup\, \tau \vL^{}_{1} \,\dotcup\,
                   \tau \vL^{}_{2} \,\dotcup\, \tau \vL^{}_{3}\ts .
\end{split}
\end{equation}
This system of equations with expanding functions, often called a
\emph{matrix function system}, induces another function system, the so
called \emph{adjoint matrix function system}, which is an iterated
function system (IFS) \cite[Ch.~5]{Bernd}. It has a unique solution --
the prototiles $T^{}_{i}$.
\begin{equation}
\begin{split}
  T^{}_{0} &\,=\, \tau^{-1} T^{}_{3} \ts ,\\
  T^{}_{1} &\,=\, \tau^{-1}  T^{}_{2} + \twovec{1}{0} \,\cup\,
                   \tau^{-1}  T^{}_{3} \ts ,\\
  T^{}_{2} &\,=\, \tau^{-1}  T^{}_{1} + \twovec{0}{1} \,\cup\,
                   \tau^{-1}  T^{}_{3} \ts ,\\
  T^{}_{3} &\,=\, \tau^{-1}  T^{}_{0}+\twovec{1}{1} \,\cup\,
  \tau^{-1}  T^{}_{1} +\twovec{0}{1} \,\cup\,\tau^{-1}
  T^{}_{2}+\twovec{1}{0} \,\cup\, \tau^{-1}  T^{}_{3}\ts .
\end{split}
\end{equation}

The set $\vL$ can be obtained as a cut and project set following the
same steps as above. The lattice $\cL$ can be understood as Minkowski
embedding of $L=\ZZ[\tau] \times \ZZ[\tau]$, which reads
\begin{equation}\label{eq:lattice}
    \cL \, =\,  \ZZ \!\begin{pmatrix} \tau \\ 0 \\ \tau' \\ 0 \end{pmatrix} \, 
    \oplus\, \ZZ\! \begin{pmatrix} 1 \\ 0 \\ 1 \\ 0 \end{pmatrix} \, 
    \oplus\, \ZZ \!\begin{pmatrix} 0\\ \tau \\ 0 \\ \tau'  \end{pmatrix} \, 
    \oplus\, \ZZ \!\begin{pmatrix} 0 \\ 1 \\ 0 \\ 1  \end{pmatrix} \, = 
    \, \cL^{}_{1} \otimes \begin{pmatrix} 1 \\ 0 \end{pmatrix} \, 
    \oplus \,\cL^{}_{1} \otimes \begin{pmatrix} 0 \\ 1 \end{pmatrix} . 
\end{equation}

The projections $\pi$, $\pi_{\!_\perp} \colon \RR^4 \to \RR^2$ defined (for a
lattice point) via
\[
  \pi \begin{pmatrix}a\tau + b \\ c\tau +d \\
    a\tau' + b \\ c\tau' +d \end{pmatrix} \, = \,
    \begin{pmatrix}a\tau + b \\ c\tau +d\end{pmatrix}, \qquad 
    \pi_{\!_\perp} \begin{pmatrix}a\tau + b \\ c\tau +d \\
      a\tau' + b \\ c\tau' +d \end{pmatrix} \, = \,
    \begin{pmatrix}a\tau' + b \\ c\tau' +d\end{pmatrix} 
\]
and the star map
$^{\star}: \pi(\cL) \to \pi_{\!_\perp}(\cL) \eqdef L^{\star}$ acting as
$\twovec{a\tau + b}{c\tau +d} \mapsto \twovec{a\tau' + b}{c\tau' +d}$
constitute the cut and project scheme. This allows us to write $\vL$
as a cut and project set with a suitable window, namely
$\vO = \bigcup^{3}_{i=0} \vO_i$. This can be obtained from the
expanding function system \eqref{eq:ssr} by considering the star image
of it and taking the closure of the lifted sets, namely
$\vO^{}_{i} = \overline{\vL^{\star}_{i}}$. The resulting iterated
function system reads
\begin{equation}\label{eq:windowIFS}
  \begin{split}
    \vO^{}_{0} &\,=\, \sigma \vO^{}_{3} + \twovec{\sigma}{\sigma}\ts ,\\
  \vO^{}_{1} &\,=\, \sigma \vO^{}_{2} + \twovec{0}{\sigma} \,\cup\,
                 \sigma \vO^{}_{3} + \twovec{0}{\sigma}\ts ,\\
  \vO^{}_{2} &\,=\, \sigma \vO^{}_{1} + \twovec{\sigma}{0} \,\cup\,
                 \sigma \vO^{}_{3} + \twovec{\sigma}{0}\ts ,\\
  \vO^{}_{3} &\,=\, \sigma \vO^{}_{0} \,\cup\, \sigma \vO^{}_{1} \,\cup\,
                 \sigma \vO^{}_{2} \,\cup\, \sigma \vO^{}_{3}\ts.
  \end{split}
\end{equation}
Since it is a contraction (on $(\cK \RR^2)^4$, see \cite{Wicks_frac}
for notation and details), it has a unique solution by Banach's
contraction mapping principle. It can be verified easily that
\begin{equation}\label{eq:windows}
\begin{aligned}
  \vO^{}_{0} &= [-1,\tau-2] \times [-1,\tau-2] , &
  \vO^{}_{1} &= [\tau-2,\tau-1] \times [-1,\tau-2]\ts , \\
  \vO^{}_{3} &= [\tau-2,\tau-1] \times [\tau-2,\tau-1] ,
  & \vO^{}_{2} &= [-1,\tau-2] \times [\tau-2,\tau-1]\ts . 
\end{aligned}
\end{equation}
Then, 
\[ 
   \vL\, \subsetneq \, \oplam (\vO) \, = \, 
   \{ \pi(x) : x\in\cL, \ \pi_{\!_\perp}(x) \in \vO \} \ts ,
\]
where $\oplam (\vO)$ is a regular model set \cite{TAO}.  Note that the
inclusion is proper due to the position of $\vO$ in the internal
space. This results in a small subtlety with the boundaries of the
windows, similar to what we saw in \eqref{eq:win1DfiboA} versus
\eqref{eq:win1DfiboB}, which we suppress here. Since the set $\vO$
satisfies $\overline{\vO^{\circ}} = \vO$ and has a boundary of measure
0, we are working with \emph{regular} model sets only.

\section{Sheared tiling and topological conjugacy}

\begin{figure}
  \centerline{\includegraphics[width=0.7\textwidth]{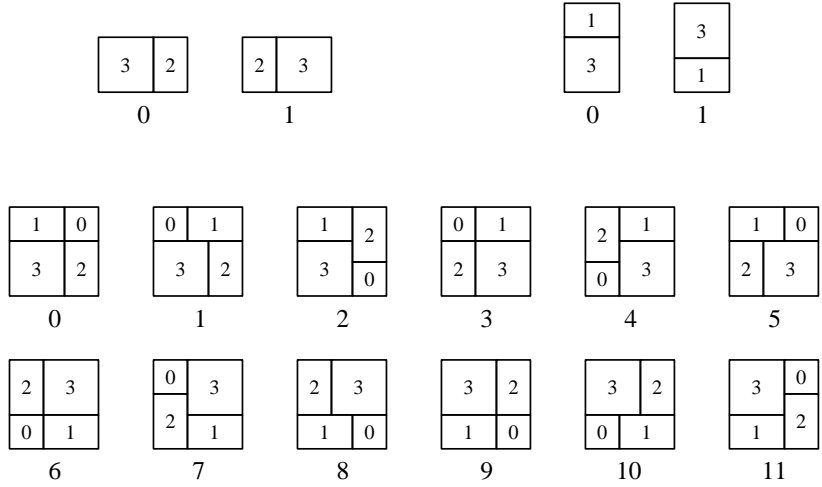}}
  \caption{Labels for the possible decompositions of the prototiles of
    type $1$ and $2$ (top row), and for the $12$ decompositions of the
    prototile of type $3$ (bottom rows).}
   \label{fig:scrambled}
\end{figure}

The original stone inflation rule \eqref{eq:fibosqrule} can be
modified. Indeed, there is a certain degree of freedom in how one can
rearrange level-1 supertiles of given prototiles. The resulting
tilings are referred to as \emph{direct product variation tilings}
(DPV tilings) and were introduced in \cite{Nat-primer,FR}, and studied
in \cite{BFG} for the Fibonacci case. Let us recall the
parametrisation used there. To each tiling, a triple of numbers
$(i^{}_{1},i^{}_{2},i^{}_{3})$ with $i^{}_{1},i^{}_{2}\in\{0,1\}$ and
$i^{}_{3}\in\{0,1,\ldots,11\}$ is assigned, based on the inflation
rules shown in the Figure \ref{fig:scrambled}. All 48 cases share the
same substitution matrix $M$ and, based on the diffraction spectra and
the equivalence theorem for pure point diffraction versus dynamical
spectra \cite{BL}, are measure-theoretically isomorphic as follows.

\begin{theorem}[\cite{BFG}, Theorem 5.2]
  The\/ $48$ inflation tiling dynamical systems that emerge from the
  above DPVs all have pure point dynamical spectrum, namely\/
  $L^{\circledast} \nts\times L^{\circledast}$, where
  $L^{\circledast} = \ZZ[\tau]/ \sqrt{5}$. These systems are thus
  measure-theoretically isomorphic by the Halmos--von Neumann theorem.
  Each individual tiling, via the control points, leads to a Dirac
  comb with pure point diffraction measure. \qed
\end{theorem}

The analysis of the 48 cases in \cite{BFG} has led to dividing them
into two different types based on the shape of their windows. The
first one consists of all DPVs with polygonal window (Figures
\ref{fig:squares} and \ref{fig:polygons}), whereas the second one of
all DPVs with fractal-like window (Figure \ref{fig:fractals}). Note
that all resulting tilings are MLD with regular model sets for the
same lattice $\cL$, but with different windows. We shall discuss the
relation later.

\begin{figure}[t]
  \centerline{\includegraphics[width=0.66\textwidth]{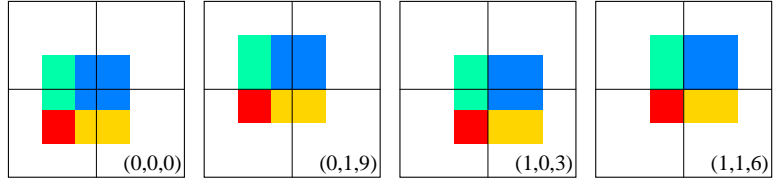}}
  \caption{Four square windows corresponding to different DP
    tilings. The underlying inflation rules form a~single orbit under
    the action of the group $D_4$, and the resulting tilings belong to
    one MLD class, because the corresponding tilings simply are
    translates of one another, hence related by a local derivation
    rule.}\label{fig:squares}
\end{figure}

\begin{figure}[t]
  \centerline{\includegraphics[width=0.99\textwidth]{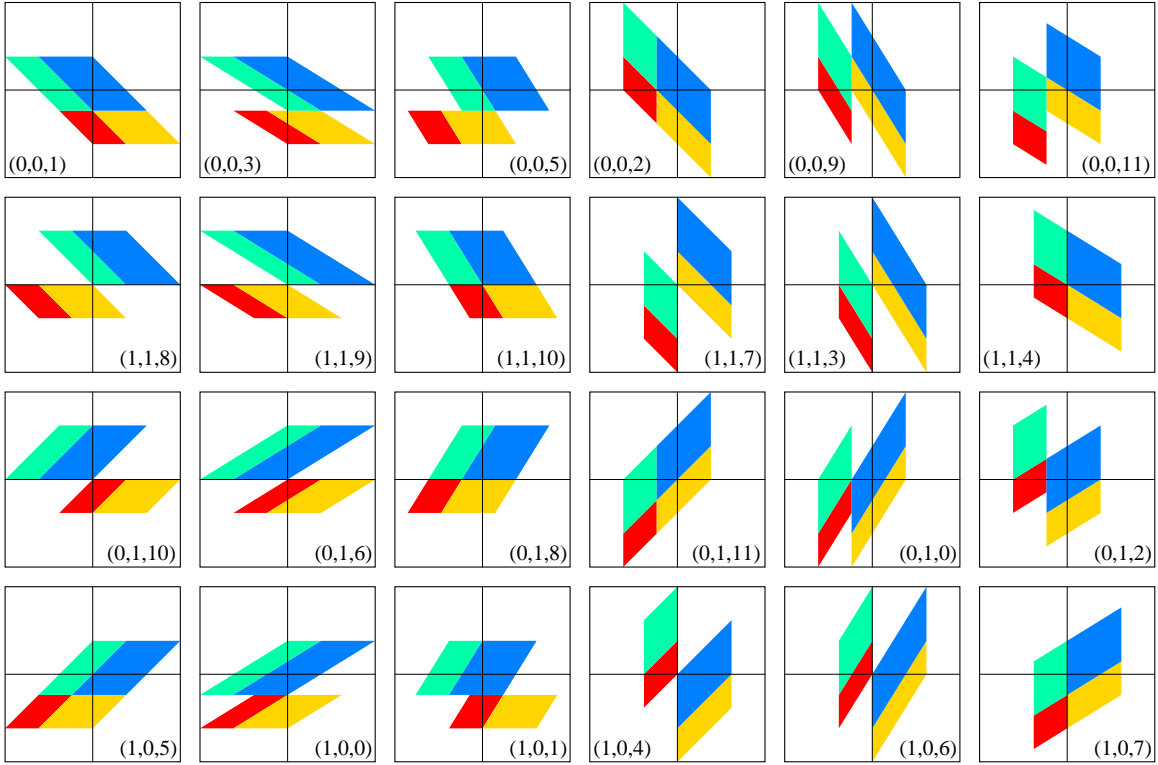}}
  \caption{Further 24 polygonal windows corresponding to different DPV
    tilings. One can recognise 3 orbits of inflation rules under the
    action of the group $D_4$, namely those of the elements $(0,0,1)$
    (first and fourth column), $(0,0,3)$ (second and fifth column) and
    $(0,0,5)$ (third and sixth column). Moreover, the 12 obvious pairs
    of windows (with equal slope) representing 12 different MLD
    classes can be recognised.}
  \label{fig:polygons}
\end{figure}

The polygonal windows can be further divided into two main classes.
One, with the square windows, corresponds to the four possible direct
product tilings (Figure \ref{fig:squares}) and lie in the same LI
class \cite{BFG}. The remaining class with 24 cases, as depicted in
Figure~\ref{fig:polygons}, can be further understood as three orbits
under the action of the dihedral group $D_4$. This action is clearly
recognisable at the level of inflation rules, but becomes less obvious
at the level of windows. One would expect that the action of the group
$D_4$ in the direct space will have its counterpart in the internal
space -- a parallelogram is mapped to another parallelogram under the
action. This is not true, as can be seen by comparing cases $(0,0,1)$
and $(1,1,8)$. These tilings are related by a~rotation through $\pi$
while the windows need some additional rearrangement. The origin for
this is our choice of the control points in each tile, which is not
invariant under the action of $D_4$. While the point set obtained from
the tiling $(0,0,1)$ via a rotation through $\pi$ is not the one given
by the control points of tiling $(1,1,8)$, it is MLD with this point
set. This can be undone by applying a local rule (changing the control
points in each type of tile).  These local rules consist of a set of
translations $t^{}_i \in \ZZ[\tau]^{2}$ of $\vL^{}_i$, which result in
a~set of translation $t^{\star}_i \in \ZZ[\tau]^{2}$ of $\vO^{}_i$ in
the internal space.

\begin{figure}[t]
  \centerline{\includegraphics[width=0.85\textwidth]{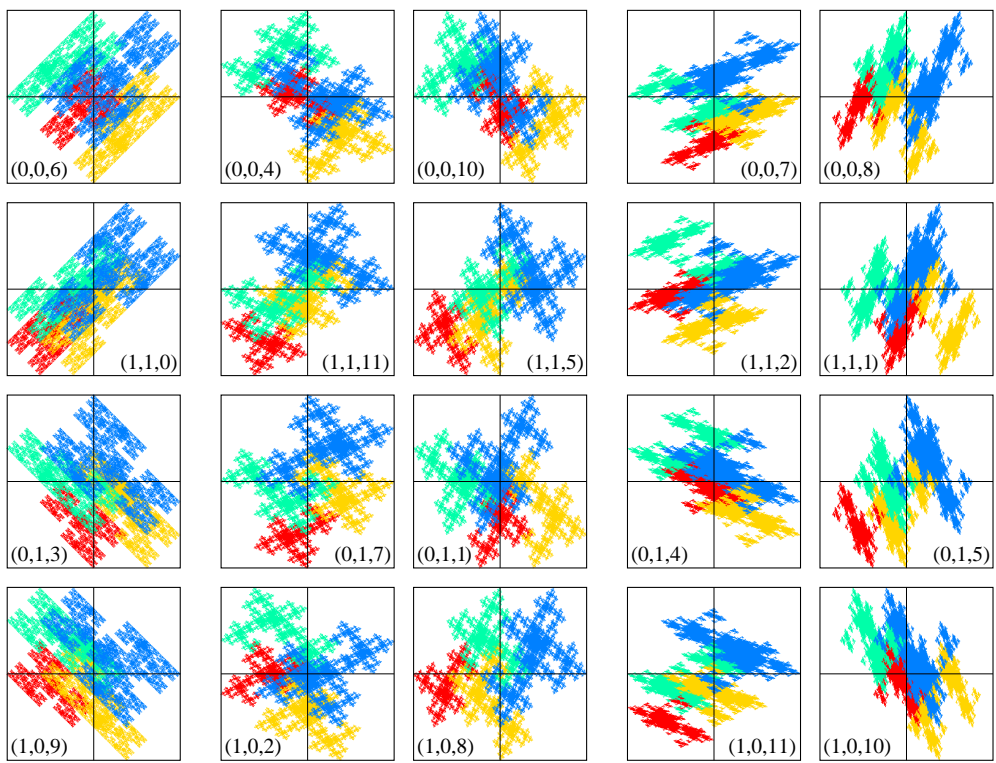}}
  \caption{The 20 fractal windows corresponding to different DPV
    tilings. The first column -- ``castle" -- corresponds to an orbit
    of rule $(0,0,6)$ under the action of $D_4$. Since the
    inflation rule has a mirror symmetry, the resulting orbit consist
    of four elements. The second and third column display an orbit of
    rule $(0,0,4)$ -- ``cross". The last two columns correspond to
    the orbit of rule $(0,0,7)$ -- type ``island".}
  \label{fig:fractals}
\end{figure}

Based on the general MLD criterion, see \cite[Rem.~7.6]{TAO} or
\cite{BSJ}, one can further divide the 24 tilings with sheared
parallelogram windows into 12 MLD classes. Two tilings in each class
are related by a rotation through $\pi$, which is a product of two
reflections, namely across coordinate axes. As mentioned above, the
corresponding windows are also related by this rotation (after a
change of the control points). This shows that the two tilings in each
MLD class are elements of \emph{the same tiling hull}. This hull then
possesses rotational symmetry.

Since there are six different shearing angles $\varphi$, namely 
\begin{equation}\label{eq:angles}
  \varphi \in \left\{ \pm \myfrac{\pi}{4},  \,
    \pm \arctan(\tau), \,  \pm \arctan\left(\tau^{-1}\right) \right\},
\end{equation}
in two possible directions (along the $x$-axis, and the $y$-axis
respectively), the MLD classification based on the slope of the
windows is clear.

\section{Relations via topological conjugacy}

Two tilings that are MLD clearly define two dynamical systems that are
topologically conjugate. The converse, however, is not true in general
because we are not in a symbolic setting. In particular, there is no
analogue of the Curtis--Hedlund--Lyndon theorem. In fact, local
derivability is the concept replacing it, but now only representing
\emph{one} possibility of topological conjugacy. Therefore, some DPV
tilings that are \emph{not} MLD could still give topologically
conjugate dynamical systems. This is indeed the case, as we first
demonstrate with an example.

Let us focus on the tiling with rule $(0,0,1)$. We will show that the
dynamical system defined by this tiling is topologically conjugate
with a DP tiling. This result should not be surprising. It follows
from the general MLD criterion that two tilings are MLD if and only if
one can obtain one window from the other by using just a \emph{finite}
number of Boolean operations (intersection, union, complement) and
translations by elements from $\pi_{\!_\perp}(\cL) =
\ZZ[\tau]^{2}$. Having a~DP tiling (with a square window), one can use
it to approximate tiling $(0,0,1)$ with increasing precision. The
general MLD criterion ensures that all the approximating tilings are
MLD. If the vertices of the sheared window belong to
$\pi_{\!_\perp}(\cL) = \ZZ[\tau]^2$, the consideration above suggests
that the tiling with the sheared window may no longer be MLD with the
DP tiling, but remains topologically conjugate to it; see
Figure~\ref{fig:approx} for an illustration.

\begin{figure}[t]
  \centering
   \centerline{\includegraphics[width=0.84\textwidth]{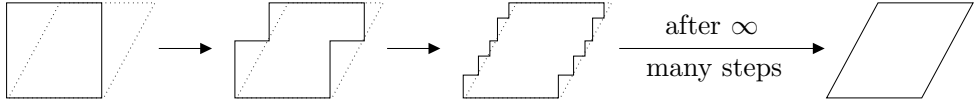}}
    \caption{Approximation of a parallelogram using a square and
      finitely many Boolean operations. This suggests topological
      conjugacy under some additional constraints on the vertex points
      of the parallelogram.}
    \label{fig:approx}
\end{figure}
  
If two tilings are MLD, there exists a~set of local derivation rules
in both directions (meaning that, from the knowledge of one tiling on
a~uniformly bounded neighbourhood of any point, one can construct the
other tiling at this point). As approximants approach the tiling with
the sheared window, the diameter of the required neighbourhood grows
without bound and, in the limit, one needs to know the whole tiling in
order to construct the sheared tiling at any given place. Thus the
locality is broken. On the other hand, it is natural to ask whether
there is some weaker relation between the tiling with sheared window
and the DP tiling. The answer is affirmative, and the relation is
\emph{topological conjugacy}, as we demonstrate next.

For the tiling with rule $(0,0,1)$, one can use a different inflation
rule (which is no longer a~stone inflation) that defines the same
tiling. This (crucial) step relies on the fact that the original
tiling has a \emph{striped structure}, where each ``row" is nothing
but a 1D Fibonacci tiling. Following the standard procedure described
for example in \cite{Dirk}, or by solving the following iterated IFS,
respectively the adjoint IFS to the modified rule\footnote{Taking the
  adjoint IFS from the tiling $(0,0,1)$ directly results in the set of
  square and rectangle prototiles.}, one obtains a set of new
prototiles~$P^{}_i$ that satisfy
\begin{equation}
\begin{split}
  P^{}_{0} &\,=\, \tau^{-1} P^{}_{3} \ts ,\\
  P^{}_{1} &\,=\, \tau^{-1}  P^{}_{2} + \twovec{1}{0} \,\cup\,
                   \tau^{-1}  P^{}_{3} \ts ,\\
  P^{}_{2} &\,=\, \tau^{-1}  P^{}_{1} + \twovec{-1}{1} \,\cup\,
                   \tau^{-1}  P^{}_{3} \ts ,\\
                   P^{}_{3} &\,=\, \tau^{-1}  P^{}_{0}+\twovec{0}{1}
                   \,\cup\, \tau^{-1}  P^{}_{1} +\twovec{-1}{1}
                   \,\cup\,\tau^{-1}  P^{}_{2}+\twovec{1}{0}
                   \,\cup\, \tau^{-1}  P^{}_{3}\ts .
\end{split}
\end{equation}
They turn the rearrangement of rule $(0,0,1)$ into a stone inflation
that is MLD with the rearranged $(0,0,1)$ rule, and thus with the
$(0,0,1)$ tiling itself; see Figure \ref{fig:shearing} for an
illustration. Note that the new tiles $P^{}_i$ are related to the
original tiles $T^{}_i$ by the shearing matrix
$S = \left(\begin{smallmatrix} 1 & -1 \\0 &
    1 \end{smallmatrix}\right)$. For obvious reasons, we call this
tiling \emph{a sheared DP tiling}. So, the $(0,0,1)$ rearrangement is
MLD with a sheared $(0,0,0)$ tiling, but \emph{not} with $(0,0,0)$
itself. However, $(0,0,1)$ and $(0,0,0)$ can give rise to
topologically conjugate systems, as we show below.

\begin{figure}
  \centering
   \centerline{\includegraphics[width=0.84\textwidth]{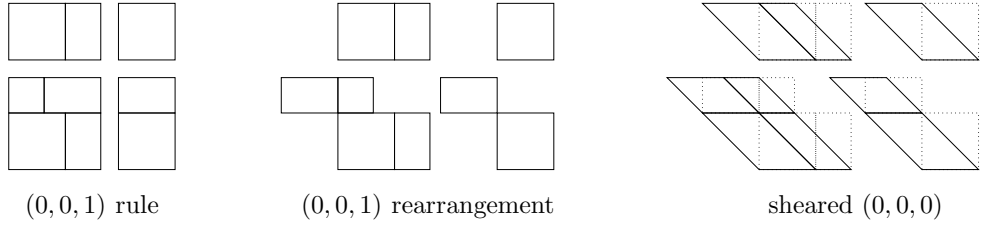}}
    \caption{The stone inflation rule $(0,0,1)$ and the rule in the
      middle produce the same tiling. The rearranged rule $(0,0,1)$
      and the sheared $(0,0,0)$ rule define two tilings that are
      MLD. We show that the sheared $(0,0,0)$ tiling and the DP tiling
      $(0,0,0)$ are topologically conjugate and thus prove the
      topological conjugacy between the tiling spaces for $(0,0,0)$
      and $(0,0,1)$.}
    \label{fig:shearing}
\end{figure}

If one denotes by $\vL_\bullet$ the chosen fixed point of the DP
tiling $(0,0,0)$ and by $\vS_\bullet$ the matching fixed point of the
sheared tiling, there is a clear correspondence via
$\vS_\bullet = S\vL_\bullet$. The procedure described above can be
applied to all 24 tilings with the polygonal window from
Figure~\ref{fig:polygons}.

\begin{fact}
  In the case of Fibonacci DPVs, each tiling with a polygonal window
  as in Figure~$\ref{fig:polygons}$ is MLD to a sheared DP tiling with
  a shearing matrix\/ $S$ of the form
\begin{equation}
\begin{pmatrix}
1 & \alpha \\ 0&1
\end{pmatrix}\ \mbox{or} \ \begin{pmatrix}
1 & 0 \\ \alpha & 1
\end{pmatrix}, \quad \mbox{each with } \alpha \in 
\left\{ \pm 1,\ \pm \tau,\ \pm\tau^{-1} \right\}, 
\label{eq:sh_mat}
\end{equation}
which corresponds to the angles from Eq.~\eqref{eq:angles}. \qed
\end{fact}

The shearing map $S$ acting in direct space has its counterpart in
internal space, which we call $S^{\star}$. The entries of this matrix
are the images of the corresponding entries of $S$ under the field
automorphism $'$. Since $\tau$ is a unit in $\ZZ[\tau]$, the following
observations hold.

\begin{fact}
  For the lattice\/ $\cL$, one has\/
  $\left(S \oplus S^{\star}\right) \cL = \cL $. Consequently,
  $S\vL_\bullet$ is a regular model set for the lattice\/ $\cL$, with
  a subset of\/ $S^{\star}\nts\vO$ of full measure as its window.  \qed
\end{fact}

Note that the window in this statement contains the entire interior of
$S^{\star}\nts\vO$, but only part of its boundary. Since $S$ is an invertible
linear mapping, we also get the following result.

\begin{fact}
  The mapping\/ $S$ is a bijection on\/ $\RR^2$, and\/ $S^{\star}$ is
  a bijection on\/ $L^{\star}$. In particular, if\/ $\vS = S\vL$,
  then\/ $0 \in \vL$ if and only if\/ $0 \in \vS$. Moreover, $\vL$ is
  a generic model set if and only if\/ $\vS$ is a generic model set,
  i.e., $\partial \vO \cap L^{\star} = \varnothing$ if and only if\/
  $\partial (S^{\star}\vO) \cap L^{\star} = \varnothing$.  \qed
\end{fact}

In order to proceed, one has to define dynamical systems induced by
the tilings $\vL_\bullet$ and $\vS_\bullet$. As usual, one
defines the geometric hulls
\begin{equation}
   \XX \,=\, \overline{\big\{\vL_\bullet +t : t \in \RR^2 \big\}}^{\mathrm{LT}} 
   \quad \mbox{and} \quad
   \YY \,=\, \overline{\big\{\vS_\bullet +t : t \in \RR^2 \big\}}^{\mathrm{LT}},
\end{equation}
where the closure is taken in the \emph{local topology}, which is
metrisable \cite{Boris}. In such a metric $\dd$, two tilings are close
if they agree on a large ball centred at the origin, possibly after
a~small global translation of one of them. Formally, two tilings
$\vL$ and $\widetilde{\vL}$ are $\varepsilon$-close,
$\dd (\vL, \widetilde{\vL}) < \varepsilon$, if
\[ 
  \vL \,\cap\, \overline{B_{\sfrac{1}{\varepsilon}}(0)} \, = \,
  (\widetilde{\vL} +t) \,\cap\, \overline{B_{\sfrac{1}{\varepsilon}}(0)}
  \quad \mbox{for some }  t\in B_{\varepsilon}(0) \ts . 
\]
Note that the metric $\dd$ is not translation invariant. 

Both hulls are equipped with the action of $\RR^2$ via
translation. This turns $\XX$ and $\YY$ into a~pair of topological
dynamical systems, $(\XX,\RR^2)$ and $(\YY,\RR^2)$. Since the defining
inflation rules are both primitive, the resulting tilings are
repetitive which is equivalent to the statement that the dynamical
hulls $\XX$ and $\YY$ are minimal. Our aim is to prove that these
dynamical systems are topologically conjugate, which is to say that
there exists a homeomorphism $\phi \colon \XX \longrightarrow \YY$
which preserves the action $g^{}_{t}$ of the group $\RR^2$ via
translation.  Then, $\phi$ is a \emph{topological conjugacy} (in the
strict sense) if $g^{}_t \circ \phi = \phi \circ g^{}_t$ holds for all
$t\in\RR^2$. In other words, the diagram
\begin{equation}\label{eq:diagram}
\begin{CD}
\XX @>g^{}_t>> \XX\\
@V{\phi}VV @VV{\phi}V\\
\YY @>g^{}_t>> \YY
\end{CD} 
\end{equation}
is commutative.

The strategy from here goes as follows. It is not difficult to see
that all patterns with polygonal windows are MLD with a sheared DP
tiling. In a first step, in each row (or column) of the tiling, all
tiles are replaced by equally sheared versions, and one then observes
that this tiling with sheared tiles is actually MLD to a sheared DP
tiling, which in turn is combinatorially equivalent to the plain DP
tiling. It is important here that the applied shear preserves the
projected lattice $\pi(\cL)$, and thus lifts to an automorphism of
$\cL$. While this shear of the DP tiling is \emph{not} an MLD
operation, it is still a topological conjugacy of the underlying
dynamical system. To show this, we select a suitable pair of tilings
and construct a bijection between the respective translation orbits
that commutes with the translation action. Then, invoking the common
parametrisation from the projection method, we show via a somewhat
subtle limit argument that this bijection between the two orbits
extends to a homeomorphism of the hulls with the same properties.

To proceed, we need to recall several results. First, since
$\vL_\bullet$ is a regular model set that arises from the cut and
project scheme $(\RR^2,\, \RR^2,\, \cL)$ with $\cL$ defined in
\eqref{eq:lattice}, there exists an $\RR^2$-invariant surjective
continuous mapping $\beta_{\XX}$, called \emph{the torus
  parametrisation}, which maps the hull $\XX$ onto the \emph{torus}
$\TT(\vL_\bullet)$. This torus is defined for our set $\vL_\bullet$ as
a~factor group, $\TT(\vL_\bullet)\defeq (\RR^2 \times \RR^2) /
\cL$. The mapping $\beta_{\XX}$ is, under our assumptions, 1--1 almost
everywhere, but not everywhere; see \cite[Cor.~5.2]{Moody_Strungaru}
or \cite{BLM} for a detailed exposition.

As already mentioned above, the cut and project scheme for set
$\vS_\bullet$ is \emph{exactly} the same as the one for the set
$\vL_\bullet$. Therefore, we can employ the same torus. This fact is
useful for our further discussion. We provide this connection by
choosing an appropriate torus parametrisation for the two hulls. In
particular, both tilings $\vL_\bullet$ and $\vS_\bullet$ are mapped to
$0$. Note that, since $\vL_\bullet$ is a singular element of $\XX$, we
have
\[ 
   \vL_{\bullet} \, \subsetneq \, \oplam (\vO) \, = \,
   \{x \in \ZZ[\tau]^2 \ : x^{\star} \in \vO \}. 
\]
From the minimality of the hulls $\XX$ and $\YY$ and from the choice
of the torus parametrisation, it follows that one can now choose
\emph{generic} elements
\begin{equation}\label{eq:def}
    \vL^{}_{\circ} \in \XX \quad \mbox{and}\quad \vS^{}_{\circ}  \in \YY
\end{equation}
such that they are mapped to the same point on the torus,
\[
  \beta_{\XX}(\vL^{}_{\circ} ) \, = \,
  \beta_{\YY}(\vS^{}_{\circ} ) \, = \, \vartheta \, \neq \, 0 \ts .
\] 
The point $\vartheta \in \RR^4$ is fixed and chosen so that
$\vL^{}_{\circ} = \oplam (\vO_{\vartheta})$, with
$ \vO_{\vartheta} \defeq \vO + \pi_{\!_\perp}\!(\vartheta)$.
Moreover, the minimality implies that
$\XX = \overline{\vL^{}_{\circ} +\RR^2\ts }^{\mathrm{LT}}$ and
$\YY = \overline{\vS^{}_{\circ} +\RR^2\ts }^{\mathrm{LT}}$.

Now, we are in the position to construct the homeomorphism $\phi$
explicitly and show that it has the properties of a topological
conjugacy. The first step consists in defining the orbit mapping
\[
  \phi \colon \{\vL^{}_{\circ}  +t : t \in \RR^2 \}
  \, \longrightarrow \, \{\vS^{}_{\circ}  +t : t \in \RR^2 \}
\]
as follows: 
\begin{enumerate}\itemsep=2pt
 \item $\phi(\vL^{}_{\circ} ) = \vS^{}_{\circ} $,
 \item $\phi(\vL^{}_{\circ} + t ) = \vS^{}_{\circ} + t$ for every
      $t \in \RR^2$.
\end{enumerate}
It is clear by construction that $\phi$ is invertible on this orbit,
and that it satisfies the required commutation property,
\begin{equation}\label{eq:com}
  \bigl(g^{}_t \circ \phi\bigr) (\vL^{}_{\circ}  + T)
  \, = \, g^{}_t (\vS^{}_{\circ}  + T)
  \, = \, \vS^{}_{\circ}  + T + t \, = \,
  \phi(\vL^{}_{\circ} +T) + t \, = \,
  \bigl(\phi \circ g^{}_t\bigr) (\vL^{}_{\circ} +T) \ts .
\end{equation} 
Our torus parametrisations ensures that
$\beta_{\YY}(\phi(\vL^{}_{\circ} +t)) = \beta_{\XX}(\vL^{}_{\circ}
+t)$, i.e., the orbit of an element $\vL^{}_{\circ} \in \XX$ has the
same image on the torus as the orbit of $\phi(\vL^{}_{\circ} )$ in the
hull $\YY$.

To establish that $\phi$ has an extension to a continuous bijection
between $\XX$ and $\YY$ is more subtle. Let us first state two results
that hold for an arbitrary, generic model set $\vL$.

\begin{lemma}\label{lem:translations}
  Let\/ $\vL$ be a regular model set, with window
  $\vO^{}_{\!\vL}$. Assume\/ $\vL$ to be generic, so\/
  $\partial \vO^{}_{\!\vL} \cap \vL^{\star} = \varnothing$. Let\/
  $\vL^{}_R = \vL \cap \overline{B^{}_R(0)}$. Then, for all\/ $R>0$,
  there exists some\/ $\delta = \delta(R)$ such that, for all\/
  $t^{\star} \in \RR^2$ with\/ $\| t^{\star}\| < \delta$, one has
\[ 
  \vL_R ^{\star} \,\subset\, \vO^{}_{\!\vL} \cap
  ( \vO^{}_{\!\vL} + t^{\star} ) \ts .
\]
In particular, the claim holds for\/ $\vL = \vL_{\circ}$. 
\end{lemma}

\begin{proof}
  Fix $R>0$. Since $\vL$ is uniformly discrete and hence locally
  finite, it follows that $\vL^{}_R $ is finite, so $\vL_R^{\star}$ is
  finite as well. This assumption assures that
\[ 
  \delta \, \defeq \min_{x^{\star} \in \vL_R^{\star}} \dist
  (x^{\star}\!, \partial \vO^{}_{\!\vL} ) 
\]
is well defined. The assumption on regularity of the set $\vL$ results
in $\delta >0$. Then, for all $x^{\star} \in \vL_R ^{\star}$, one gets
$x^{\star} \in \vO^{}_{\!\vL} + t^{\star}$, and we are done.
\end{proof}

\begin{remark}
  Requiring genericity of $\vL$ can often be relaxed by replacing it
  with several constraints on the direction of the vector
  $t^{\star}$. In our case at hand, this can easily be done since the
  shape of the window $\vO$ is a~polygon and thus sufficiently simple.
  \exend
\end{remark}

One immediate consequence of Lemma~\ref{lem:translations} is the
following result.

\begin{prop}\label{prop:basic}
  Let\/ $\vL$ be a regular model set that is generic, so\/
  $\vL^{\star}\, \cap\, \partial \vO^{}_{\!\vL} = \varnothing$.
  Further, let\/ $t \in L = \ZZ[\tau]^{2}$. Then, there are\/
  $\delta^{}_1 = \delta^{}_1(\varepsilon)>0$ and\/
  $\delta^{}_2 = \delta^{}_2(\varepsilon)>0$ such that, if\/ $\vL$
  and\/ $\vL + t$ are\/ $\varepsilon$-close, then\/
  $\| t^{\star} \| < \delta^{}_1$, and, if\/
  $\| t^{\star} \| < \delta^{}_2$, then\/ $\vL$ and\/ $\vL + t$ are\/
  $\varepsilon$-close.
\end{prop}

\begin{proof}
  Denote by $A_\varepsilon$ the set of $\varepsilon$-almost periods of
  $\vL$ relative to $\dd$, so
\[
  A_{\varepsilon} \, = \, \{t : \dd(\vL, \vL+t) < \varepsilon \} \ts .
\]
Then, the statements hold if there exist $\delta^{}_1$
and $\delta^{}_2$ such that
\[
  B^{}_{\delta_2}(0) \, \subseteq \, \overline{A_{\varepsilon}^{\star}}
  \, \subseteq \, B^{}_{\delta_1}(0) \ts .
\]
Clearly, $\delta^{}_2$ can be chosen as described in
Lemma~\ref{lem:translations}. The second bound can be obtained as
$\delta^{}_1 \defeq \sup \{\, \| t \| \, :\, t\in \RR^2, \, t +
\vL^{\star}_{1/\varepsilon} \subset \vO^{}_{\!\vL} \}$.
\end{proof}

Proposition~\ref{prop:basic} can directly be applied to our
situation. Suppose that $\vL^{}_{\circ}$ and $\vL^{}_{\circ} + t$ are
$\varepsilon$-close, so
$\dd(\vL^{}_{\circ},\, \vL^{}_{\circ} + t) < \varepsilon$. This
implies that $\| t^{\star} \| < \delta(\varepsilon)$. Then, for
$S = \left(\begin{smallmatrix} 1 & \alpha \\ 0 &
    1 \end{smallmatrix}\right)$, we get
\[ 
  \big\| ( S^{-1}t )^{\star} \big\| \, = \,
  \big\| t^{\star} - \alpha^{\star} \twovec{t^{\star}_1}{0}
  \big\| \, \leqslant \, \bigl( 1+|\alpha^{\star}| \bigr)
  \| t^{\star} \| \, \leqslant \, \bigl( 1+|\alpha^{\star}| \bigr)
  \delta(\varepsilon) \ts .  
\]
Note that the same estimate holds also for any matrix $S$ of the form
$S = \left(\begin{smallmatrix} 1 & 0 \\ \alpha &
    1 \end{smallmatrix}\right)$.

Proposition~\ref{prop:basic} ensures that there is some
$\widetilde{\varepsilon}\in \cO(\varepsilon)$ such that
$\vL^{}_{\circ}$ and $\vL^{}_{\circ} + S^{-1}t$ are
$\widetilde{\varepsilon}$-close which here is equivalent to the
statement that $\vL^{}_{\circ}$ and $\vL^{}_{\circ} + S^{-1}t$ agree
on $\overline{B_{\sfrac{1}{\widetilde{\varepsilon}}}(0)}$. Applying
$S$ yields that $S\vL^{}_{\circ}$ and
$S\vL^{}_{\circ} + SS^{-1}t = S\vL^{}_{\circ} + t$ agree on
$S\overline{B_{\sfrac{1}{\widetilde{\varepsilon}}}(0)}$. Since the
matrix $S$ is invertible,
$S\overline{B_{\sfrac{1}{\widetilde{\varepsilon}}}(0)}$ is an ellipse
centred at the origin. Therefore, there exists $R>0$ such that
$B_R(0) \subset
S\overline{B_{\sfrac{1}{\widetilde{\varepsilon}}}(0)}$, where the
radius $R$ actually satisfies
\[
  R \, < \, \myfrac{c}{\widetilde{\varepsilon}}
  \quad \mbox{with} \quad c\, = \,
    \myfrac{1}{\sqrt{2}}\sqrt{\alpha^2 -
    \alpha\sqrt{\alpha^2+4 \ts }+2 \ts }.
\]
Note that $c$ is the smaller singular value of $S$. Taking
$\underline{\varepsilon}$ such that
$R \geqslant \frac{1}{\underline{\varepsilon}}$ results in the
following claim.

\begin{prop}\label{prop:close}
  Let\/ $\vL^{}_{\circ}$ be the regular, generic model set from\/
  \eqref{eq:def}, so\/
  $\vL^{\star}_{\circ} \, \cap\, \partial \vO_{\vartheta} =
  \varnothing$. Further, let\/ $t \in \ZZ[\tau]^{2}$. Then, if\/
  $\vL^{}_{\circ}$ and\/ $\vL^{}_{\circ} + t$ are\/ $\varepsilon$-close,
  $\vS^{}_{\circ}$ and\/ $\vS^{}_{\circ} + t$ are
  $\underline{\varepsilon}$-close, with\/
  $\underline{\varepsilon} \in \cO(\varepsilon)$. Formally,
  $\dd(\vL^{}_{\circ},\, \vL^{}_{\circ} + t) < \varepsilon \
  \Rightarrow \ \dd(\vS^{}_{\circ},\, \vS^{}_{\circ} + t) <
  \underline{\varepsilon}\ts .$   \qed
\end{prop}

\begin{remark}
  Since the set $\vL$ is aperiodic, the set of its periods is trivial,
  so
\[
   \vL +t \,= \, \vL \quad \Longleftrightarrow \quad t \, = \, 0 \ts , 
\]
and the following implication holds for
$(t_n)^{}_{n\in\NN}$ with $t_n \in \ZZ[\tau]^2$,
\begin{equation}\label{eq:aper_impl}
  \lim_{n\to  \infty} \vL + t_n = \vL \quad
  \Longrightarrow \quad \lim_{n\to  \infty} t^{\star}_n = 0 \ts .
\end{equation}
This has an immediate consequence, namely $(t_n )^{}_{n\in\NN}$ is
eventually 0, or $t_n$ has a subsequence that grows
unboundedly. Therefore, if
\[
   \lim_{n\to \infty} \vL + t_n \, = \lim_{n\to \infty} \vL + s_n 
   \quad \mbox{for } t_n,s_n \in \ZZ[\tau]^2, 
\]
then $(t_n-s_n)^{}_{n\in\NN}$ is eventually 0, or unboundedly growing
such that $\lim_{n\to \infty} (t_n-s_n)^{\star} = 0$. Note that the
converse of \eqref{eq:aper_impl} is not true due to the existence of
singular elements in the hull $\XX$. Such elements always exist when
$\vL$ is aperiodic \cite{BLM}.  \exend
\end{remark}

Proposition \ref{prop:close} is a good starting point for our proof of
the continuity of $\phi$, as $\vS^{}_{\circ} = \phi(\vL^{}_{\circ})$
and $\vS^{}_{\circ} + t = \phi(\vL^{}_{\circ}+t)$. Since the distance
$\dd$ is not translation invariant, one has to work with converging
sequences. Thus, our next aim is to show that, if
$\vL^{}_{\circ} + t_n$ approaches $\vL^{}_{\circ}$ in $\XX$, then
$\phi(\vL^{}_{\circ} + t_n)$ approaches $\phi(\vL^{}_{\circ})$ in
$\YY$.

\begin{lemma}\label{cor:coro1}
  Let\/ $\vL^{}_{\circ}$ be the regular, generic model set
  from~\eqref{eq:def}. If\/ $\vL^{}_{\circ} + t_n$ approximates
  $\vL^{}_{\circ}$ in $\XX$ with $t_n \in \RR^2$, then\/
  $\phi(\vL^{}_{\circ} + t_n)$ approximates\/ $\phi(\vL^{}_{\circ})$ in
  $\YY$, formally
\[ 
  \lim_{n\to \infty} \dd(\vL^{}_{\circ},  \vL^{}_{\circ}  + t_n) = 0
  \, \Longrightarrow  
  \lim_{n\to \infty} \dd\bigl(\phi(\vL^{}_{\circ}),
  \phi(\vL^{}_{\circ} + t_n)\bigr) = 0 \ts .  
\]
\end{lemma}

\begin{proof}
  We distinguish between two cases. First, let us assume that the
  sequence $(t_n )^{}_{n\in\NN}$ consists of elements of
  $ \ZZ[\tau]^{2}$ only. Then, the claim holds due to Proposition
  \ref{prop:close}.

  Thus, suppose that the $t_n$ lie in $\RR^{2}$. Since
  $\ZZ[\tau]^{2}$ is dense in $\RR^2$, there exists a sequence
  $(\delta_n)_{n\in \NN}$ with $\delta_n \in \RR^2$ such that
  $t_n + \delta_n \in \ZZ[\tau]^2$ and $\|\delta_n\| < \frac{1}{2^n}$
  for all $n\in\NN$ . Since $\vL^{}_{\circ} + t_n$ and
  $\vL^{}_{\circ} + t_n +\delta_n$ agree on the whole plane up to a
  translation by $\delta_n$, they are $\frac{1}{2^n}$-close. Using the
  triangle inequality, one obtains
\[ 
  \dd(\vL^{}_{\circ},  \vL^{}_{\circ} + t^{}_n + \delta^{}_n)
   \, \leqslant \,  \dd(\vL^{}_{\circ}, \vL^{}_{\circ} + t^{}_n) 
   + \dd(\vL^{}_{\circ} + t^{}_n, \vL^{}_{\circ} + t^{}_n + \delta^{}_n)
   \, < \, \dd(\vL^{}_{\circ},  \vL^{}_{\circ} + t^{}_n) + \myfrac{1}{2^n} \ts . 
\]
Therefore,
$ \dd(\vL^{}_{\circ}, \vL^{}_{\circ} + t^{}_n + \delta^{}_n) \xrightarrow{n
  \to \infty} 0 $, and we can use the first part of this claim since
$t_n + \delta_n \in \ZZ[\tau]^2$. Thus,
\[
  \lim_{n\to + \infty} \dd\bigl(\phi(\vL^{}_{\circ}),
  \phi(\vL^{}_{\circ} + t_n + \delta_n) \bigr) \, = \, 0 \ts . 
\]
From the fact that $\vL^{}_{\circ} + t_n$ and
$\vL^{}_{\circ} + t_n +\delta_n$ are $\frac{1}{2^n}$-close, it follows
(by construction) that also $\phi(\vL^{}_{\circ} + t_n)$ and
$\phi(\vL^{}_{\circ} + t^{}_n + \delta^{}_n) = \phi(\vL^{}_{\circ} +
t^{}_n) +\delta_n$ are $\frac{1}{2^n}$-close. Thus, using the triangle
inequality again, one concludes that
\[ 
  \dd\bigl(\phi(\vL^{}_{\circ}), \phi(\vL^{}_{\circ} + t^{}_n )\bigr)
  \,\leqslant \, \dd\bigl(\phi(\vL^{}_{\circ}),
  \phi(\vL^{}_{\circ} + t^{}_n + \delta^{}_n ) \bigr) + 
  \myfrac{1}{2^n} \, \xrightarrow{n \to \infty} \, 0 \ts , 
\]
which proves our claim.
\end{proof}

\begin{lemma}\label{cor:coro2}
  Let\/ $T \in \RR^2$. Then, for the regular, generic model set\/
  $\vL^{}_{\circ}$ from~\eqref{eq:def}, one has
\[ 
  \lim_{n\to \infty} \dd(\vL^{}_{\circ} +T , \vL^{}_{\circ} + t^{}_n )=0
  \, \Longrightarrow \lim_{n\to \infty} \dd \bigl(\phi(\vL^{}_{\circ}+T ),
  \phi(\vL^{}_{\circ} + t_n) \bigr) \, = \, 0 \ts .  
\]
\end{lemma}

\begin{proof}
  Suppose that
  $\lim_{n\to \infty} \dd(\vL^{}_{\circ} +T , \, \vL^{}_{\circ} +
  t_n)=0$. Then, $\vL^{}_{\circ} +T$ and $\vL^{}_{\circ} + t^{}_n$ are
  $\varepsilon_n$-close which means that they coincide on a ball
  $\overline{B_{\sfrac{1}{\varepsilon_n}}(0)}$ up to a small
  translation. This is equivalent to the statement that
  $\vL^{}_{\circ}$ and $\vL^{}_{\circ} +t^{}_n -T$ agree on a ball
  $\overline{B_{\sfrac{1}{\varepsilon_n}}(T)}$. From the assumption,
  one has $\varepsilon_n \to 0$ as $n\to\infty$ and thus there exists
  $n^{}_0$ with $0 \in \overline{B_{\sfrac{1}{\varepsilon_n}}(T)}$ for
  all $n\geqslant n^{}_0$. Then,
  $B_{\sfrac{1}{\varepsilon'_n}}(0) \subset
  \overline{B_{\sfrac{1}{\varepsilon_n}}(T)}$ with
  $\frac{1}{\varepsilon'_n} = \frac{1}{\varepsilon_n} - \|T\|$. Hence
  $\varepsilon'_n = \frac{1}{\sfrac{1}{\varepsilon_n}-\|T\|}$, which
  shows that $\varepsilon'_n \to 0$ as $n\to
  \infty$. Lemma~\ref{cor:coro1} now implies the claim.
\end{proof}

At this point, one has to extend the mapping $\phi$ to the closure of
the orbit of $\vL^{}_{\circ}$, i.e., to the hull $\XX$. Note that this
step profits from the minimality of $\XX$ and from the fact that we
may use the same torus for $\XX$ and $\YY$. Recall that
$\partial \vO_{\vartheta} \cap L^{\star} = \varnothing$ if and only if
$\partial (S^{\star}\vO_{\vartheta}) \cap S^{\star}L^{\star} =\partial
(S^{\star}\vO_{\vartheta}) \cap L^{\star} = \varnothing$. This allows
us to identify the fibres corresponding to singular points in each
hull. Thus, if a sequence $(\vL^{}_{\circ} + t_n )^{}_{n\in\NN}$
converges towards a singular element $\widetilde{\vL} \in \XX$, our
choice ensures that, \emph{if} $(\vS^{}_{\circ} + t^{}_n )^{}_{n\in\NN}$
converges in $\YY$, the resulting limit point $\widetilde{\vS}$ is
also singular. This would allow us to \emph{continuously} extend the
mapping $\phi$ to the hull $\XX$ by defining
$\phi(\widetilde{\vL}) \defeq \lim_{n \to \infty} \vS^{}_{\circ} +
t^{}_n = \widetilde{\vS}$.

The following lemma shows that the above consideration about the
convergence holds.

\begin{lemma}\label{lem:extension}
  Let\/ $\vL^{}_{\circ}$ be the regular, generic model set
  from~\eqref{eq:def}. Suppose that we have\/
  $ \vL^{}_{\circ} + t^{}_n \xrightarrow{n \to \infty} \widetilde{\vL}$
  with\/ $\widetilde{\vL}\in \XX$. Then, the sequence\/
  $(\vS^{}_{\circ} + t^{}_n )^{}_{n\in\NN}$ converges in $\YY$.
\end{lemma}

\begin{proof}
  If $\vL^{}_{\circ} +t^{}_n$ converges towards a generic element
  $\widetilde{\vL}$, then
  $\beta^{}_{\XX}(\vL^{}_{\circ} +t^{}_n) \xrightarrow{n \to \infty}
  \beta_{\XX}(\widetilde{\vL})$. Since
  $\beta^{}_{\XX}(\vL^{}_{\circ} +t^{}_n) = \beta^{}_{\YY}(\vS^{}_{\circ}
  +t^{}_n)$ and
  $\beta^{-1}_{\YY}\bigl( \{\beta^{}_{\XX}(\widetilde{\vL})\}\bigr)$
  contains only one element, the claim follows immediately.

  Now, suppose that the limit point $\widetilde{\vL}$ of
  $\vL^{}_{\circ} +t_n$ is a singular element of $\XX$. We show the
  claim by a contradiction. For this purpose suppose that
  $\lim_{n \to \infty} \vS^{}_{\circ} + t^{}_n$ does not exist, i.e.,
  the sequence $(\vS^{}_{\circ} + t^{}_n )_{n\in\NN}$ has at least two
  accumulation points. Without loss of generality, one may assume that
  $(\vS^{}_{\circ} + t^{}_n )^{}_{n \in\NN}$ has exactly two
  accumulation points, say $\vS_1$ and $\vS_2$ with
  $\vS_1 \neq \vS_2$. Then, there exist two subsequences of
  $(t_n)^{}_{n\in\NN}$, say
  $(s_n)^{}_{n\in\NN} = (t_{k_n})^{}_{n\in\NN}$ and
  $(r_n)^{}_{n\in\NN} = (t_{\ell_n})^{}_{n\in\NN}$ with
  $\vS_1 = \lim_{n \to \infty} \vS^{}_{\circ} + s_n$, and
  $\vS_2 = \lim_{n \to \infty} \vS^{}_{\circ} + r_n$,
  respectively. These subsequences can be chosen so that
  $\{s_n \, : \, n \in \NN\} \cap \{r_n \, : \, n \in \NN\} =
  \varnothing$.

  Since the limit points are not equal, there exists a position where
  they differ. This allows us to find a ball centred at the origin on
  which $\vS_1$ and $\vS_2$ do not coincide. Thus, there is an index
  $n^{}_0$ with
\begin{equation}\label{eq:difference}
  \bigl( (\vS^{}_{\circ} + s^{}_{n_0}) \triangle
  (\vS^{}_{\circ} + r^{}_{n_0}) \bigr) \cap B_{n_0}(0)
  \, \neq \, \varnothing 
\end{equation}
with $\triangle$ standing for the symmetric difference of two
sets. Set $N \defeq \max\, \{k_{n_0},\ell_{n_0}\}$ and define
\[
   \widetilde{s}_n = 
    \begin{cases}
        s_n, & \text{if $n\leqslant n_0$,}\\
        t^{}_N, & \text{otherwise,}
    \end{cases} \qquad \mbox{and} \qquad 
    \widetilde{r}_n = 
    \begin{cases}
        r_n, & \text{if $n\leqslant n_0$,}\\
        t^{}_N, & \text{otherwise.}
    \end{cases} 
\]
By construction, one has
$\vL^{}_{\circ}+\widetilde{s}^{}_n \xrightarrow{n \to
  \infty}\vL^{}_{\circ}+t^{}_N$ and
$\vL^{}_{\circ}+\widetilde{r}^{}_n \xrightarrow{n \to \infty}
\vL^{}_{\circ}+t^{}_N$. Lemma~\ref{cor:coro2} now implies that
$\vS^{}_{\circ}+ \widetilde{s}^{}_n \xrightarrow{n \to \infty}
\vS^{}_{\circ}+t^{}_N$ and
$\vS^{}_{\circ}+\widetilde{r}^{}_n \xrightarrow{n \to \infty}
\vS^{}_{\circ}+t^{}_N$. In particular, this convergence gives
$\bigl( (\vS^{}_{\circ} + s_{n_0}) \triangle (\vS^{}_{\circ} + t^{}_N)
\bigr) \cap B_{n_0}(0) = \varnothing $, and
$\bigl( (\vS^{}_{\circ} + r_{n_0}) \triangle (\vS^{}_{\circ} + t^{}_N)
\bigr) \cap B_{n_0}(0) = \varnothing $ which together with inclusion
\[ 
  (\vS^{}_{\circ} + s_{n_0}) \triangle (\vS^{}_{\circ} + r_{n_0})
  \, \subset \, \bigl((\vS^{}_{\circ} + s_{n_0}) 
  \triangle (\vS^{}_{\circ} + t_N)\bigr) \cup \bigl(
  (\vS^{}_{\circ} + r_{n_0}) \triangle (\vS^{}_{\circ} + t_N) \bigr) 
\]
contradicts \eqref{eq:difference}. 
\end{proof}

It is an easy exercise to show that the extension $\phi$ commutes with
the translation using \eqref{eq:com} together with a sequence
$(t_n)^{}_{n\in\NN}$ and taking the limit. The inverse is also
well-defined due to Lemma \ref{lem:extension}. Let us summarise the
obtained results as follows.

\begin{prop}
  Let\/ $\vL^{}_{\circ} \in \XX$ and $\vS^{}_{\circ}\in \YY$ be as
  defined in\/ \eqref{eq:def}. Then, the mapping\/ $\phi$ defined
  by
\begin{enumerate}\itemsep=2pt
\item $\phi(\vL^{}_{\circ} + t ) \defeq \vS^{}_{\circ} + t$ for every
  $t \in \RR^2$ and
\item if
  $\widetilde{\vL} = \lim_{n\to \infty}\vL^{}_{\circ} + t^{}_n$, then
  $\phi(\widetilde{\vL}) \defeq \lim_{n\to \infty}\vS^{}_{\circ} +
  t^{}_n$
\end{enumerate}
is a homeomorphism between\/
$\XX = \overline{\vL^{}_{\circ} + \RR^2 \ts }^{\mathrm{LT}}$ and\/
$\YY = \overline{\vS^{}_{\circ}+ \RR^2 \ts }^{\mathrm{LT}}$ that commutes
with the group action\/ $g^{}_t$. In other words, $\phi$ is a
topological conjugacy.  \qed
\end{prop}

Each tiling with a polygonal window (as in Figure \ref{fig:polygons})
is MLD to a sheared DP tiling with a matrix $S$ of the form
\eqref{eq:sh_mat}. This tiling, in turn, is topologically conjugate to
a direct product tiling (with the square window, see Figure
\ref{fig:squares}). Since all DP tilings are MLD, they are also
topologically conjugate, and since the topological conjugacy is a
transitive relation, we conclude that any two Fibonacci DPV tilings
with polygonal windows are topologically conjugate. We are now able to
state our main result as follows.

\begin{theorem}\label{thm:new}
  The\/ $28$ inflation tiling dynamical systems that emerge from the
  above DPVs with polygonal windows form one class of topologically
  conjugate dynamical systems. \qed
\end{theorem}

Since the DPVs with fractal windows of different type can not be
topologically conjugate (due to distinct Hausdorff dimensions of the
window's boundaries), the classification up to topological conjugacy
is essentially complete. As mentioned in the beginning, the situation
is more complex than in the case of fully periodic structures. Our
above analysis was still feasible because all tilings with polygonal
windows were either DP tilings or striped versions thereof. Since this
need no longer hold for more complicated DP tilings, it is clear that
other phenomena are to be expected then. This seems an interesting
problem for further work.

\section*{Acknowledgements}

It is our pleasure to thank Natalie Priebe Frank and Lorenzo Sadun for
discussions and useful suggestions. This work was supported by the
German Research Foundation (DFG) within the CRC 1283/2 (2021 -
317210226) at Bielefeld University.


\end{document}